\documentclass[notitlepage,leqno,11pt]{article}

\usepackage{amssymb}
\usepackage{amsfonts}
\usepackage{amscd}
\usepackage{amstext}
\usepackage{latexsym,amssymb}
\usepackage{amsmath}
\usepackage{graphicx}
\usepackage{color}
\usepackage[all]{xy}
\usepackage{lineno}

\newcommand{\R}{\mathbb{R}}

\newcommand{\N}{\mathbb{N} }

\newcommand{\Chi}{\raise 1.5pt\hbox{$\chi$}}

\newcommand{\beq}{\begin{equation}}
\newcommand{\eeq}[1]{\label{#1}\end{equation}}
\newcommand{\beqar}{\begin{eqnarray*}}
\newcommand{\eeqar}{\end{eqnarray*}}

\newcommand{\proof}[1]{\par\smallskip\noindent{\bf Proof: }}
\newcommand{\qed}{\penalty 500\hfill$\square$\par\medskip}

\newcommand\interruptenum{\xdef\savecounter{\theenumi}\end{enumerate}}
\newcommand\continueenum{\begin{enumerate}%
\setcounter{enumi}{\savecounter}\item[]}

\newcommand{\proofof}{}

\renewcommand{\text}[1]{\hbox{#1}}
\renewcommand{\eqref}[1]{(\ref{#1})}

\newcommand{\avg}{\raise 0pt\hbox{$-$}\hskip -10.7pt\int}

\newcounter{stepcounter}

\newtheorem{lem}         {Lemma}[section]
\newtheorem{pro}    [lem]{Proposition}
\newtheorem{thm}    [lem]{Theorem}
\newtheorem{rem}    [lem]{Remark}
\newtheorem{cor}    [lem]{Corollary}
\newtheorem{df}     [lem]{Definition}

\date{}

\title{On Moment Condition and Center Condition for Abel Equation}

\begin{document}


\maketitle


\begin{center}
{{\large Anderson L. A. de Araujo, Ab\'ilio Lemos and Alexandre Miranda Alves} \\\vspace{.5cm} Universidade Federal de Vi\c{c}osa, CCE, Departamento de Matem\'atica\\ Avenida PH Rolfs, s/n\\ CEP 36570-900, Vi\c{c}osa, MG, Brasil \\ E-mails: \tt anderson.araujo@ufv.br, abiliolemos@ufv.br, amalves@ufv.br}
\end{center}

%

\

\noindent{\sc Abstract}. In this paper we consider  Abel equation $x' = g(t)x^2+f(t)x^3$, where $f$ and $g$ are analytical functions.  We proved that if  the equation has a center at $x=0$, then the Moment Conditions, i. e., $m_k=\int_{-1}^1f(t)(G(t))^kdt=0,~~k=0,1,2$, is satisfied where $G(t)=\int_{-1}^tg(s)ds$. Besides, we give partial a positive answer to a conjecture proposed by Y. Lijun and T. Yun in 2001.

\



\

\newpage

\section{Introduction}

\subsection{Historical Aspect}

Let the planar system

\begin{equation}\label{S1}
\displaystyle{\begin{array}{ccc}
\dot{x} & = &  -y+P(x,y)\\
 \dot{y} & = &  x+Q(x,y),\\
\end{array}}
 \end{equation}
where $P(x,y)$ and $Q(x,y)$ are polynomials, without constant term, of maximum degree n. The singular point $(0,0)$ is a center, if surrounded by closed trajectories; or a focus, if surrounded by spirals. The classical center-focus problem consists in distinguishing when a  singular point is either a center or a focus. The problem started with Poincar\'e \cite{P} and Dulac \cite{Du}, and, in the present days, many questions remain unanswered. The basic results were obtained by A. M.
Lyapunov \cite{Li}. He proved that if $P(x,y)$ and $Q(x,y)$ satisfy an infinite sequence of recursive conditions, then \eqref{S1} has a center to the origin. He also presented conditions for the origin of the system \eqref{S1} to be a focus.

If we write $P(x,y)=\sum_{i=1}^{l}P_{m_i}(x,y)$ and $Q(x,y)=\sum_{i=1}^{l}Q_{m_i}(x,y)$, where $P_{m_i}(x,y)$ and $Q_{m_i}(x,y)$ are homogeneous polynomials of degree $m_i\geq1$, then, from Hilbert's theorem on the finiteness of basis of polynomial ideals (\cite{Ka}, Theorem 87, p. 58), it follows that, in the mentioned infinite sequence of recursive conditions, only a finite number of conditions for center are essential. The others result from them.  

Let us consider a particular case of \eqref{S1}. Namely,

\begin{equation}\label{eq.1}
\displaystyle{\begin{array}{ccc}
\dot{x} & = &  -y+P_n(x,y)\\
 \dot{y} & = &  x+Q_n(x,y),\\
\end{array}}
 \end{equation}
where $P_n(x,y)$ and $Q_n(x,y)$ are homogeneous polynomials of degree n.  

When $n = 2$, systems \eqref{eq.1} are quadratic polynomial differential systems (or simply
quadratic systems in what follows). Quadratic systems have been intensively studied over
the last 30 years, and more than a thousand papers on this issue have been published (see,
for example, the bibliographical survey of Reyn \cite{Re}).

A method for investigating if \eqref{eq.1} has a center at the origin consists in transforming the planar system into an Abel equation. 
In polar coordinates $(r,\theta)$ defined by $x = r\cos\theta, y = r\sin\theta$, the system (\ref{eq.1}) becomes
\begin{equation}\label{eq.2}
\displaystyle{\begin{array}{ccl}
\dot{r} & = &  A(\theta)r^n\\
 \dot{\theta} & = &1+B(\theta)r^{n-1},\\
\end{array}}
 \end{equation}
where 
\begin{equation}\label{eq.3}
\displaystyle{\begin{array}{ccl}
A(\theta) & = &  \cos\theta P_n(\cos\theta,\sin\theta) + \sin\theta Q_n(\cos\theta,\sin\theta),\\
B(\theta) & = & \cos\theta Q_n(\cos\theta,\sin\theta) - \sin\theta P_n(\cos\theta,\sin\theta).\\
\end{array}}
 \end{equation}
We remark that $A$ and $B$ are homogeneous polynomials of degree $n + 1$ in the variables
$\cos\theta$ and $\sin\theta$. In the region
$$R=\{(r,\theta):1+B(\theta)r^{n-l}>0\},$$
the differential system (\ref{eq.2}) is equivalent to the differential equation

\begin{equation}\label{eq.4}
\displaystyle{\frac{dr}{d\theta} =\frac{A(\theta)r^n}{1+B(\theta)r^{n-1}}}.
\end{equation}
It is known that the periodic orbits surrounding the origin of the system (\ref{eq.2}) do not intersect the curve $\theta = 0$ (see the Appendix of \cite{CL}). Therefore, these periodic orbits are contained in the region $R$. Consequently, they are also periodic orbits of equation (\ref{eq.4}).

The transformation $(r, \theta) \rightarrow (\gamma, \theta)$ with

\begin{equation}\label{eq.5}
\displaystyle{\gamma =\frac{r^{n-1}}{1+B(\theta)r^{n-1}}}
\end{equation}
is a diffeomorphism from the region R into its image. As far as we know, Cherkas was the first to use this transformation (see \cite{C}). If we write equation (\ref{eq.4}) in the variable $\gamma$, we obtain
\begin{equation}\label{eq.6}
\displaystyle{\frac{d\gamma}{d\theta} =-(n-1)A(\theta)B(\theta)\gamma^3 + [(n-1)A(\theta)-B'(\theta)]\gamma^{2}},
\end{equation}
which is a particular case of an Abel differential equation. We notice that $f(\theta)=-(n-1)A(\theta)B(\theta)$ and $g(\theta)=(n-1)A(\theta)-B'(\theta)$ are homogeneous trigonometric polynomials of
degree $2(n + 1)$ and $n + 1$, respectively.

Now the Center-Focus problem of 
equation \eqref{S1} has a translation in equation \eqref{eq.6}. That is, given $\gamma_0$ small enough, we
look for necessary and sufficient conditions on $f(\theta)$ and $g(\theta)$ in order to assure
that the solution of equation \eqref{eq.6} with the initial condition $\gamma(0) = \gamma_0$ has the
property that $\gamma(0) = \gamma(2\pi)$. We observe that this condition implies the periodicity of this solution.

In the present paper, we consider a certain variant of the Center-Focus problem related to the original one - the Center-Focus problem for the Abel differential equation
\begin{equation}\label{Abel1}
	\frac{dx}{dt} = g(t)x^2+f(t)x^3 
\end{equation}
This problem is to provide necessary and sufficient conditions on $f$ and $g$ and on $a,b \in \R$ for all the solutions $x(t)$ of \eqref{Abel1} to satisfy $x(a)=x(b)$. When $f$ and $g$ are polynomial functions, the equation \eqref{Abel1} is called polynomial Abel equation. 
Notice that, now, such condition does not imply the periodicity of $x(t)$. In recent years, the Center-Focus problem for the Polynomial Abel equation has advanced substantially, as observed in \cite{BFY2}, \cite{BFY3}, \cite{BRY},  \cite{Br1}, \cite{Br2}, \cite{Ch},  \cite{Yo}.

\subsection{Composition conjecture}

\begin{df}\label{1.1}
An Abel differential equation \eqref{Abel1} is said to have a center at $a\leq t\leq b$ if $x(a)=x(b)$ for any solution $x(t)$ (with the initial value $x(a)$ small enough), or equivalently, the equation \eqref{Abel1} has a center at $x=0$ if all the solutions
close to $x=0$ are closed.
\end{df}

The Center-Focus problem is to give necessary and sufficient conditions on
$f$, $g$ for the Abel equation above to have a center. The only, known to us, sufficient condition for the Center is the following
''Polynomial Composition Condition'' (PCC):

\begin{df} The polynomials $f=F'$, $g=G'$ are said to satisfy the Polynomial
Composition Condition (PCC) on $[a,b]$, if there exist polynomials
$\widetilde{F}$, $\widetilde{G}$ and $W$, such that
\[F(t)=\widetilde{F}(W(t)), G(t)=\widetilde{G}(W(t)), W(a)=W(b).\]
\end{df}

(PCC) is also known to be necessary for the Center for small degrees of $f$, $g$ and in some other very special situations. Notice that (PCC) is described by a finite number of algebraic equations on the coefficients of $f$, $g$. 

\begin{pro}\label{prop5} If $f$, $g$ satisfy the Polynomial Composition Condition (PCC)
on $A=[a,b]$, then the Abel equation (\ref{Abel1}) has a center on $A$.
\end{pro}
\begin{proof}  

Indeed, after a change of variables $w= W(t)$, we obtain a new polynomial Abel equation
\begin{equation}\label{Abel2}
\frac{d\widetilde{y}}{dw} = \widetilde{f}(w)\widetilde{y}^3+\widetilde{g}(w)\widetilde{y}^2
\end{equation}
with $\widetilde{f}=\widetilde{F}'$, $\widetilde{g}=\widetilde{G}'$. All the solutions of (\ref{Abel1}) are obtained from the solutions of
(\ref{Abel2}) by the same substitution $y(t)=\widetilde{y}(W(t))$. Fulfilling $W(a)=W(b)$, one has $y(a)=y(b)$.
\end{proof}

All existing results in the literature so far supports the following conjecture: 

\vspace{.5cm}
\textbf{COMPOSITION CONJECTURE}. \textit{The polynomial Abel equation (\ref{Abel1}) has a center on the set of points $A=[a,b]$ if and only if the Composition Condition (PCC) holds for $f$ and $g$ on $A$.}

\vspace{.5cm}
This conjecture has been verified for small degrees of $f$ and $g$ and in many
special cases in \cite{BBY}, \cite{Bl1}, \cite{Bl2}, \cite{BFY1}, \cite{BFY2}, \cite{BFY3}, \cite{BFY4}, \cite{BFY5},  \cite{Ch}, \cite{YT}, \cite{Yo}.

The equation \eqref{Abel1} was studied in \cite{AGG}, where necessary and sufficient conditions were obtained for this equation to have a center at the origin, where $f(t)$ and $g(t)$ are particular continuous functions. More results that ensure the existence of a center at the origin for some subclasses of Abel equations were obtained in \cite{LLLZ,LLZ}.

When $f$ and $g$ are odd polynomials, then $f(t)=t\hat{f}(t^2)$ and $g(t)=t\hat{g}(t^2)$. Results on  sufficient conditions presented below were obtained by Alwash and Lloyd \cite{AL}.
\begin{pro}[\cite{AL}]\label{lloyd}
Assume $f, g \in C([a,b])$ to be expressed by 
\begin{equation}\label{Composition}
	f(t)=\hat{f}(\sigma(t))\sigma'(t), \,\, g(t)=\hat{g}(\sigma(t))\sigma'(t)
\end{equation}
for some continuous functions $\hat{f}, \hat{g}$ and a continuously differentiable function $\sigma$, which is closed, i.e., $\sigma(a)=\sigma(b)$. Then, the Abel equation
\[
x'(t)=f(t)x^3(t)+g(t)x^2(t), \,\, t \in [a,b]
\]
has a center in $[a,b]$.
\end{pro}

If true, the Composition Conjecture tells us a lot about the nature of the return map for Abel equations and its relationship with the coefficients of the system.
However, it also highlights a significant difference between the polynomial and trigonometric cases. In the latter case, it is known that the class corresponding to the Composition Conjecture (that is, those systems with $P$ and $Q$ polynomials of a trigonometric polynomial), although a significant class, does not exhaust all
possible center conditions, see \cite{AL2}.


\subsection{Moment conditions and the Parametric composition conjecture}
In this section, we consider the polynomial Abel differential equation
\begin{equation}\label{par1}
	x'(t)=g(t)x^2(t)+\epsilon f(t)x^3(t), \,\, t \in [a,b],
\end{equation}
where $x$ is real, $\epsilon \in \R$ and $g(t)$ and $f(t)$ are real polynomials. Let us assume that $\int_a^bg(s)ds=0$.

One of the issues that can be tackled is characterizing when \eqref{par1} has a center in
$[a,b]$ for all $\epsilon$ with $|\epsilon|$ small enough. This type of centers are called \textit{infinitesimal centers}
or \textit{persistent centers}, see \cite{AL2}, \cite{CGM}. In \cite{AL2}, it was proved that a necessary and sufficient conditions for \eqref{par1} to have a center in $[a,b]$  are 
\begin{equation}\label{par2}
	\int_a^bf(t)(G(t))^kdt=0,
\end{equation}
for all natural numbers $k \in \N\cup \{0\}$. Conditions \eqref{par2} are called the \textit{moment conditions}. The composition conjecture for moments is that the \textit{moments conditions} imply the \textit{composition condition}. Moreover, in \cite{BRY} it is proved that ``at infinity'' the center conditions are reduced to the moment conditions.

A counterexample to the \textit{composition conjecture for moments} in the polynomial case
was given in \cite{Pa}, see too \cite{GGS}. 

In the trigonometric case, that is, if one considers a trigonometric Abel differential equation
of the form
\begin{equation}\label{par3}
	x'(\theta)=g(\theta)x^2(\theta)+\epsilon f(\theta)x^3(\theta), \,\, \theta \in [0,2\pi],
\end{equation}
where $x$ is real, and $\epsilon$ is a real value close to $0$. One can define the composition conjecture for moments analogously
to the polynomial case. The moment conditions in this case are written as 
\begin{equation}\label{par4}
	\int_0^{2\pi}f(\theta)(G(\theta))^kd\theta=0,
\end{equation}
for all natural numbers $k \in \N\cup \{0\}$. It is also possible to construct a counterexample of the composition conjecture
for moments, see \cite{GGS}

In Lijun and Yun \cite[Theorem 5.5]{LY}, the authors proved the following result.
\begin{pro}\label{A2}
Let $f(\cdot)$ be a polynomial. Then, for $\epsilon$ small enough the Abel equation
\[
x'=2tx^2 + \epsilon f(t)x^3, \,\,\, t \in [-1,1]
\] 
has a center $x=0$ if and only if $f(t)$ is an odd polynomial.
\end{pro}
Also, the authors conjecture that the conclusion of Theorem \ref{A2} holds without $\epsilon$. 

\subsection{Main results}

This paper aims to study the following Abel's equation
\begin{equation}\label{A1}
	x'(t)=f(t)x^3+g(t)x^2,~t \in [-1,1],
\end{equation}
 where $f$ and $g$ are analytic functions. Now, we state our main results.

\begin{thm}\label{T1}
Consider the Abel's equation
\begin{equation*}
	x'=f(t)x^3+g(t)x^2,~t \in [-1,1],
\end{equation*}
 where $f$ and $g$ are analytic functions. If  this equation has a center at $x=0$, then 
\[
m_k=\int\limits_{-1}^1f(t)(G(t))^kdt=0,~~k=0,1,2,
\]
where $G(t)=\int\limits_{-1}^tg(s)ds$.
\end{thm}
\begin{rem}
The following equation, 
\[
\frac{dx}{d\theta}=(\sin\theta - \sin2\theta + \sin3\theta)x^3+(\cos\theta + 2\cos2\theta)x^2,~\theta \in [0,2\pi],
\]
studied by Gin\'e, Grau and Santallusia \cite{GGS2}, has a center at $x=0$. Some computations show that $m_0=0$, $m_1=0$, $m_2=0$ and 
\[
m_3=\int\limits_0^{2\pi}(\sin\theta - \sin2\theta + \sin3\theta)\left(\int\limits_0^{\theta}(\cos\,s + 2\cos2s)ds\right)^3d\theta=\frac{\pi}{2}\neq 0, 
\]
which demonstrates that the result of Theorem \ref{T1} is the best possible. 
\end{rem}

Now, considering $f $ and $g$ real polynomial functions in  Abel's equation \eqref{A1}, we obtain the following result. 

\begin{thm}\label{T1.2}
Consider the Abel's equation
\begin{equation}\label{abell}
	x'(t)=f(t)x^3+g(t)x^2,~t \in [-1,1],
\end{equation}
 where $f(t)=\sum_{j=0}^da_jt^j=p(t^2)+tq(t^2)$ is such that $p(t^2)$ changes sign at most two times in $[-1,1]$ and $g(t)=t^{n-1}$, where $n$ is a positive even integer. If  the equation \eqref{abell} has a center at $x=0$, then $$m_k=\int_{-1}^1f(t)(G(t))^kdt=0,~~k=0,1,2,3,4\ldots,$$
where $G(t)=\int_{-1}^tg(s)ds=\frac{1}{n}(t^n-1)$.
\end{thm}

If $g(t)=t^{n-1}$ is as in Theorem \ref{T1.2}, we have the following result. 
\begin{cor}\label{T4}
Let $f(\cdot)$ be a polynomial of degree $d\leq 5$. Then, the Abel equation
\begin{equation}\label{E3.5}
	x'=t^{n-1}x^2 + f(t)x^3, \,\,\, t \in [-1,1],
\end{equation}
has a center $x=0$ if and only if $f(t)$ is an odd polynomial. 
\end{cor}

\begin{rem}\label{AT1}
The results of Theorem \ref{T1.2} and Corollary \ref{T4} hold if $g(t)=nt^{n-1}$, where $n$ is a positive even integer. Hence, $G(t)=\int_{-1}^tg(s)ds=t^n-1$.
\end{rem}
According to Corollary \ref{T4} and Remark \ref{AT1}, we give a positive answer to the conjecture proposed by Lijun and Yun \cite[Remark 5.6]{LY} when $f(\cdot)$ is a polynomial of degree $d\leq 5$ and $n=2$.

The following diagram gives us some relations and implications about the \textit{Center problem},  \textit{Moment conditions} and the \textit{Composition Conjecture}.

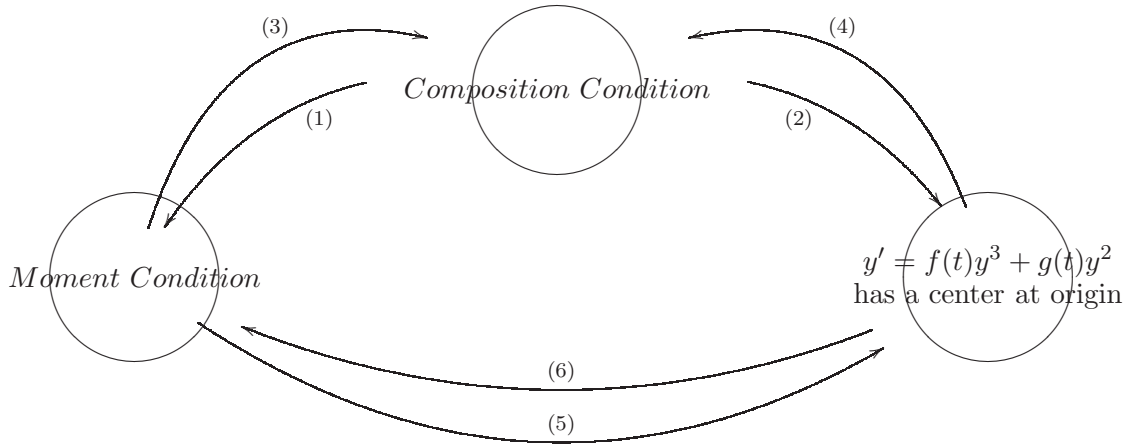
\vspace{3cm}

\xymatrix{
& *+++++[o][F-]{Composition~ Condition}  \ar@/_1.2cm/[dl]^{(1)}
 \ar@/^1.2cm/[dr]_{(2)} &\\
 *+++++[o][F-]{ Moment~ Condition}  \ar@/_2.2cm/[rr]^{(5)}  \ar@/^2.0cm/[ur]^{(3)}& & *+++++[o][F-] \txt {$y'=f(t)y^3+g(t)y^2$ \\ has a center at origin}  \ar@/_2.0cm/[ul]_{(4)}  \ar@/^1.5cm/[ll]_{(6)}
 }

\vspace{1cm}
\noindent
Implication $(1)$ was proved in \cite[p. 13]{BBY}.\\
Implication $(2)$ was proved in \cite[p. 442]{BRY}.\\ 
Implication $(3)$ is not generally true. For example, see \cite{Pa} and \cite{GGS}.\\ 
Implication $(4)$ is generally an open problem, namely \textit{Composition Conjecture}. Several particular cases were proved. See, for example, \cite{BFY1} and the reference therein. We proved some particular cases, see Corollaries  \ref{T4}. \\
Implication $(5)$ is generally an open problem. \\
Implication $(6)$ holds in several particular cases. One of these cases is the main result of this paper, see Theorem \ref{T1.2}.

\begin{rem}
Notice that, if $(5)$ is true, then the composition conjecture (see $(4)$) is not true.
\end{rem}


\section{Preliminaries results}
Following Yang Lijun and Tang Yun \cite{LY}, we write the expression below, for a solution $x(t,\rho)$ of the Abel equation \eqref{A1} satisfying $x(-1,\rho)=\rho$.
\begin{equation}\label{0.0}
	x(t,\rho)=\rho +\sum_{k=2}^{\infty}r_k(t)\rho^k.
\end{equation}

To prove Theorems \ref{T1} and \ref{T1.2}, we apply the following result of Yang Lijun and Tang Yun \cite[Lemma 5.2, p 108]{LY}.

\begin{lem}\label{LY}
The origin $x=0$ is a center of the Abel equation \eqref{A1} if and only if 
\begin{equation}\label{condit0}
\int_{-1}^1g(t)dt=0 \mbox{ and } \int_{-1}^1f(t)r_k(t)dt=0, \,\, k\geq 0
\end{equation}
or, equivalently, if and only if
\begin{equation}\label{condit1}
\int_{-1}^1g(t)dt=0 \mbox{ and } \int_{-1}^1f(t)x(t,\rho)dt=0, \,\, |\rho|<\rho_0
\end{equation}
for $\rho_0$ small enough, where 
\[
x(t,\rho)=\frac{\rho}{\displaystyle 1-\rho\int_{-1}^t(f(s)x(s)+g(s))ds}.
\]
\end{lem}

Now, suppose that $x=0$ is a center of the equation \eqref{A1}. Notice that a solution of \eqref{A1} is equivalent to a solution of the integral equation
\begin{eqnarray}\label{eq5}
x(t,\rho)=\frac{\rho}{\displaystyle 1-\rho\int_{-1}^t(f(s)x(s,\rho)+g(s))ds}, t \in [-1,1]
\end{eqnarray}
where $x(-1)=\rho$, for $\rho$ small enough such that $\rho\int_{-1}^t(f(s)x(s,\rho)+g(s))ds <1$ for all $t\in[-1,1]$. With the integral equation \eqref{eq5} of Abel equation we define the following nonlinear operator
\begin{eqnarray}
T_{\rho}&:&C[-1,1]\rightarrow C[-1,1]\\
T_{\rho}(x)(t)&=&\frac{\rho}{1-\rho \int_{-1}^{t}(f(s)x(s,\rho)+g(s))ds}
\end{eqnarray}
where $f,g\in C[-1,1]$ and $\rho\in \mathbb{R}$. Of course, $T_{\rho}$ is well defined on arbitrary bounded set of $C[-1,1]$ if $\rho$ is small enough. In  \cite{LY}, Lijun and Yun proved that $T_{\rho}$ is a contraction. According to the well known Banach contraction theorem, $T_{\rho}$ has a unique fixed point in $B_1=\{f\in C[-1,1];||f||\leq 1\}$, where $||f||=\sup_{t\in[-1,1]}|f(t)|$. In addition, Lijun and Yun proved that, for $\rho\leq \rho_0=(\sqrt{g}+||g||+||f||+1)^{-1}$, this fixed point is the solution $x(t,\rho)$ of the Abel equation \eqref{A1}, with $x(-1,\rho)=\rho$ and $||x||\leq 1$. 

According to the Abel equation \eqref{A1}, $f(t,x)=f(t)x^3+g(t)x^2,$ with $t \in [-1,1]$ and $x \in \R$, is analytic. Thus, the solution $x(t,\rho)$ is also analytic in $(t,\rho)$ (see in \cite[Th 8.2, p. 35]{CL2}).

Set 
\begin{equation}\label{H}
	H(t,\rho)=\int_{-1}^t(f(s)x(s,\rho)+g(s))ds.
\end{equation}
Since $f$, $g$ and $x$ are analytic functions, we have that $H$ is analytic. Furthermore, as $||x(t,\rho)||\leq 1$ and $f$ and $g$ are limited on the interval $[-1,1]$, for small enough $\rho$, we have $-1<\rho\int_{-1}^t(f(s)x(s,\rho)+g(s))ds <1$. Thus, the following identities are well defined
\begin{eqnarray*}
x(t,\rho)&=&\frac{\rho}{\displaystyle 1-\rho\int_{-1}^t(f(s)x(s,\rho)+g(s))ds}\\
&=&\rho\frac{1}{\displaystyle 1-\rho H(t,\rho)}\\
&=&\rho\left(1+H(t,\rho)\rho + H^2(t,\rho)\rho^2 + H^3(t,\rho)\rho^3 + \dots \right)\\
&=&\rho+H(t,\rho)\rho^2 + H^2(t,\rho)\rho^3 + H^3(t,\rho)\rho^4 + \dots \\
&=& \displaystyle \rho\sum_{k=0}^{\infty}H^k(t,\rho)\rho^{k}.
\end{eqnarray*}

By Lemma \ref{LY}, we obtain  
\[
\int_{-1}^1f(t)x(t,\rho)dt=0, \,\, |\rho|<\rho_0
\]
The last identities lead us to conclude that
\[
\int_{-1}^1f(t)x(t,\rho)dt=\rho\sum_{k=0}^{\infty}\int_{-1}^1f(t)H^k(t,\rho)dt\rho^{k}=0, \,\, |\rho|<\rho_0.
\]

In the proof of Theorem \ref{T1}, we need to show that $\int\limits_{-1}^1f(t)H^k(t,\rho)dt=0$ in $\rho=0$ and $k=0,1,2$. For such, we need some lemmas, which are presented below.

%
%
%

\begin{lem}\label{A6}
\[
\displaystyle \frac{\partial^{j}}{\partial \rho^{j}}\left(H(1,\rho)\right){\big|_{\rho =0}}=j!\int\limits_{-1}^1f(t)G^{j-1}(t)dt \mbox{ for each } j=1,2
\]
and
\[
\frac{\partial^{3}}{\partial \rho^{3}}\left(H(1,\rho)\right){\big|_{\rho =0}}=3!\int\limits_{-1}^1f(t)G^{2}(t)dt + 3!\int\limits_{-1}^1f(t)F(t)dt,
\]
where $G(t)=\int_{-1}^tg(s)ds$ and $F(t)=\int_{-1}^tf(s)ds$.
\end{lem}

\begin{proof}

Since
\begin{eqnarray}\label{0}
x(t,\rho)&=&\rho+H(t,\rho)\rho^2 + H^2(t,\rho)\rho^3 + H^3(t,\rho)\rho^4 + \dots 
\end{eqnarray}
we obtain
\begin{equation}\label{1}
	\frac{\partial}{\partial \rho}x(t,\rho)=1+\frac{\partial}{\partial \rho}(H(t,\rho)\rho^2) + \frac{\partial}{\partial \rho}(H^2(t,\rho)\rho^3) + \dots, 
\end{equation}
\begin{equation}\label{2}
\begin{array}{rcl}
	\frac{\partial^2}{\partial \rho^2}x(t,\rho)&=&\frac{\partial^2}{\partial \rho^2}(H(t,\rho)\rho^2) + \frac{\partial^2}{\partial \rho^2}(H^2(t,\rho)\rho^3) + \dots\\
	&=& 2H(t,\rho) + R_1(t,\rho),
	\end{array}
\end{equation}
where $R_1(t,\rho)|_{\rho=0}=0$ and
\begin{equation}\label{3}
\begin{array}{rcl}
	\frac{\partial^3}{\partial \rho^3}x(t,\rho)&=&\frac{\partial^3}{\partial \rho^3}(H(t,\rho)\rho^2) + \frac{\partial^3}{\partial \rho^3}(H^2(t,\rho)\rho^3) + \dots\\
	&=& 3!H^2(t,\rho) + 3!\frac{\partial}{\partial \rho}H(t,\rho) + R_2(t,\rho),
	\end{array}
\end{equation}
where $R_2(t,\rho)|_{\rho=0}=0$. 

By definition of $H$ and \eqref{1}, we obtain
\[
\frac{\partial}{\partial \rho}H(t,\rho)=\int\limits_{-1}^tf(s)\frac{\partial}{\partial \rho}x(s,\rho)ds
\]
and
\begin{equation}\label{4}
	\frac{\partial}{\partial \rho}H(t,\rho)|_{\rho=0}=\int\limits_{-1}^tf(s)ds=F(t). 
\end{equation}
Hence,
\begin{equation}\label{4.2}
	\frac{\partial}{\partial \rho}H(1,\rho)|_{\rho=0}=\int\limits_{-1}^1f(t)dt. 
\end{equation}

By definition of $H$ and \eqref{2}, we obtain
\[
\frac{\partial^2}{\partial \rho^2}H(t,\rho)=2\int\limits_{-1}^tf(s)H(s,\rho)ds +\int\limits_{-1}^tR_1(s,\rho)ds 
\]
and
\begin{equation}\label{5}
	\frac{\partial^2}{\partial \rho^2}H(t,\rho)|_{\rho=0}=2\int\limits_{-1}^tf(s)H(s,0)ds=2\int\limits_{-1}^tf(s)G(s)ds. 
\end{equation}
Hence,
\begin{equation}\label{5.2}
	\frac{\partial^2}{\partial \rho^2}H(1,\rho)|_{\rho=0}=2\int\limits_{-1}^1f(t)G(t)dt. 
\end{equation}

By definition of $H$, \eqref{3} and \eqref{4}, we obtain
\[
\frac{\partial^3}{\partial \rho^3}H(t,\rho)=3!\int\limits_{-1}^tf(s)H(s,\rho)ds + 3!\int\limits_{-1}^tf(s)\frac{\partial}{\partial \rho}H(s,\rho)ds +\int\limits_{-1}^tR_2(s,\rho)ds 
\]
and
\begin{equation}\label{6}
\begin{array}{rcl}
	\frac{\partial^3}{\partial \rho^3}H(t,\rho)|_{\rho=0}&=&3!\int\limits_{-1}^tf(s)H^2(s,0)ds + 3!\int\limits_{-1}^tf(s)\frac{\partial}{\partial \rho}H(s,0)ds\\
	&=& 3!\int\limits_{-1}^tf(s)G^2(s)ds + 3!\int\limits_{-1}^tf(s)F(s)ds.
	\end{array}
\end{equation}
Hence,
\[
\frac{\partial^3}{\partial \rho^3}H(1,\rho)|_{\rho=0}= 3!\int\limits_{-1}^1f(t)G^2(t)dt + 3!\int\limits_{-1}^1f(t)F(t)dt.
\]

\end{proof}

\begin{lem} \label{A7}
Let $H$ be defined by \eqref{H}, then
\[
\displaystyle \frac{\partial^{j}}{\partial \rho^{j}}\left(H(1,\rho)\right){\big|_{\rho =0}}=j!\int\limits_{-1}^1f(t)r_j(t)dt \mbox{ for each } j=1,2,3.
\]
\end{lem}

\begin{proof}

If we replace the expression \eqref{0.0} by the formula defining $H(t,\rho)$, we obtain
\[
\begin{array}{rcl}
	H(t,\rho)&=&\displaystyle \int_{-1}^t[f(s)x(s,\rho)+g(s)]ds\\
	&=&\displaystyle \left(\int_{-1}^tf(s)ds\right)\rho + \sum_{k=2}^{\infty}\left(\int\limits_{-1}^tf(s)r_k(s)ds\right)\rho^k +\int_{-1}^tg(s)ds\\
	&=&\displaystyle G(t)+ F(t)\rho + \sum_{k=2}^{\infty}\left(\int\limits_{-1}^tf(s)r_k(s)ds\right)\rho^k.
	\end{array}
\]
Hence,
\[
\displaystyle \frac{\partial^{j}}{\partial \rho^{j}}\left(H(t,\rho)\right){\big|_{\rho =0}}=j!\int\limits_{-1}^tf(s)r_j(s)ds \mbox{ for each } j=1,2,3.
\]
Therefore,
\[
\displaystyle \frac{\partial^{j}}{\partial \rho^{j}}\left(H(1,\rho)\right){\big|_{\rho =0}}=j!\int\limits_{-1}^1f(t)r_j(t)dt \mbox{ for each } j=1,2,3.
\]
\end{proof}

\section{Proof of Theorem \ref{T1}}

By Lemma \ref{LY}, we have  
\[
\int_{-1}^1f(t)r_k(t)dt=0, \,\, k\geq 0.
\]
By Lemmas \ref{A6} and \ref{A7}, we obtain
\[
\displaystyle \frac{\partial^{j}}{\partial \rho^{j}}\left(H(1,\rho)\right){\big|_{\rho =0}}=j!\int\limits_{-1}^1f(t)G^{j-1}(t)dt=j!\int_{-1}^1f(t)r_j(t)dt=0, \mbox{ for each } j=1,2
\]
and
\[
\frac{\partial^{3}}{\partial \rho^{3}}\left(H(1,\rho)\right){\big|_{\rho =0}}=3!\int\limits_{-1}^1f(t)G^{2}(t)dt + 3!\int\limits_{-1}^1f(t)F(t)dt=3!\int_{-1}^1f(t)r_3(t)dt=0.
\]
Since by Lemma \ref{LY} we obtain $F(1)=\int_{-1}^1f(t)dt=0$, we conclude that
\[
\int\limits_{-1}^1f(t)F(t)dt=\int\limits_{-1}^1F'(t)F(t)dt=\frac{1}{2}\int\limits_{-1}^1\frac{d}{dt}F^2(t)dt=\frac{1}{2}[F^2(1)-F^2(-1)]=0.
\]
Therefore,
\[
\int\limits_{-1}^1f(t)G^{2}(t)dt=\int_{-1}^1f(t)r_3(t)dt=0.
\]
Hence, we obtain
\[
m_k=\int_{-1}^1f(t)(G(t))^kdt=0, \,\, k=0,1,2.
\]

\section{Proof of Theorem \ref{T1.2}}

In the proof of Theorem \ref{T1}, we will use the result due to Briskin, Francoise and Yomdin \cite[Theorem 4.1]{BFY2}. By \cite{BFY2}, it is sufficient to analyze the sign changes of the function
\[
\psi(u)=\sum_{i=1}^msgn(g(t_i(u)))\frac{f(t_i(u))}{g(t_i(u))}
\] 
where $t_1(u), t_2(u),\dots, t_m(u)$ are the solutions in $[-1,1]$ of the equation $G(t)=u$ where $u \in [u_0,u_1]$, with $u_0=\min\limits_{[-1,1]}G(t)$ and $u_1=\max\limits_{[-1,1]}G(t)$.

Since $g(t)=t^{n-1}$, we obtain $G(t)=\frac{1}{n}(t^n-1)$. Hence, $u_0=-\frac{1}{n}$ and $u_1=0$. In particular, if $u \in [-\frac{1}{n},0]$, we obtain $0\leq 1+nu\leq 1$. The solutions in $[-1,1]$ of the equation
\[
G(t)=\frac{1}{n}(t^n-1)=u
\]
are $t_1(u)=(1+nu)^{\frac{1}{n}}$ and $t_2(u)=-(1+nu)^{\frac{1}{n}}$. Therefore, since $n$ is even,
\[
\begin{array}{rcl}
\psi(u)&=&sgn(g(t_1(u)))\frac{f(t_1(u))}{g(t_1(u))}+sgn(g(t_2(u)))\frac{f(t_2(u))}{g(t_2(u))}\\
&=& \frac{f(t_1(u))}{g(t_1(u))}-\frac{f(t_2(u))}{g(t_2(u))}\\
&=& \frac{1}{(1+nu)^{\frac{n-1}{n}}}[f((1+nu)^{\frac{1}{n}}) + f(-(1+nu)^{\frac{1}{n}})].
\end{array}
\]
Now, we define the polynomial $h(t)=f(t)+f(-t)=p(t^2)$. Since according assumptions, the polynomial $h(t)$ changes sign at most two times in $[-1,1]$, we conclude that $\psi(u)$ changes sign at most two times in $[-1,0]$.

Since the equation \eqref{abell} has a center at $x=0$, then by Theorem \ref{T1},
\[
m_k=\int\limits_{-1}^1f(t)(G(t))^kdt=0,~~k=0,1,2.
\]
Hence, $\psi(u)$ changes sign at most two times in $[-1,0]$ and $m_k=0$ for $k=0,1,2$. According to \cite[Theorem 4.1]{BFY2}, we obtain that $m_k=0$ for $k=1,2,\dots$, that is
\[
m_k=\int_{-1}^1f(t)(G(t))^kdt=0,~~k=0,1,2,\ldots,
\]
where $G(t)=\int_{-1}^tg(s)ds=\frac{1}{n}(t^n-1)$.

\section{Proof of Corollary \ref{T4}}
We can write the polynomial $f$ as
\[
f(t)=a_0+a_1t+a_2t^2+a_3t^3+a_4t^4+a_5t^5.
\]
By Theorem \ref{T1}  
\[
m_k=\int_{-1}^1f(t)(G(t))^kdt=0,~~k=0,1,2
\]
where $G(t)=\int_{-1}^tg(s)ds=\frac{1}{n}(t^n-1)$. Hence,
\[
\int_{-1}^1f(t)dt=0,
\]
\[
\frac{1}{n}\int_{-1}^1f(t)(t^n-1)dt=0
\]
and
\[
\frac{1}{n^2}\int_{-1}^1f(t)(t^n-1)^2dt=0.
\]
Since $t^n-1$ is an even function, we obtain
\[
\int_{-1}^1[a_0+a_2t^2+a_4t^4]dt=0,
\]
\[
\frac{1}{n}\int_{-1}^1[a_0+a_2t^2+a_4t^4](t^n-1)dt=0
\]
and
\[
\frac{1}{n^2}\int_{-1}^1[a_0+a_2t^2+a_4t^4](t^n-1)^2dt=0.
\]
By solving the integrals, we obtain the following system

\begin{equation}\label{S1.2}
\left\{\begin{array}{ccc}
2a_0 + \frac{2}{3}a_2 + \frac{2}{5}a_4&=&0\\
-\frac{2n}{1+n}a_0 - \frac{2n}{3(3+n)}a_2 - \frac{2n}{5(5+n)}a_4&=&0\\
\frac{4n^2}{(1+2n)(1+n)}a_0 + \frac{4n^2}{3(3+2n)(3+n)}a_2 + \frac{4n^2}{5(5+2n)(5+n)}a_4&=&0,
\end{array}
\right.
\end{equation}
or equivalently
\begin{equation}\label{S2}
\left[ \begin{array}{ccc} 1 & \frac{1}{3}  & \frac{1}{5} \\ 
\\
\frac{1}{1+n} &  \frac{1}{3(3+n)}  &  \frac{1}{5(5+n)}\\
\\
\frac{1}{(1+2n)(1+n)} & \frac{1}{3(3+2n)(3+n)}  & \frac{1}{5(5+2n)(5+n)}\\
\end{array} \right]
\left[ \begin{array}{c} a_0\\ 
a_2 \\
a_4 
\end{array} \right]=
\left[ \begin{array}{c} 0\\ 
0 \\
0
\end{array} \right].
\end{equation}

Notice that the matrix associate
\[
M=\left[ \begin{array}{ccc} 1 & \frac{1}{3}  & \frac{1}{5} \\ 
\\
\frac{1}{1+n} &  \frac{1}{3(3+n)}  &  \frac{1}{5(5+n)}\\
\\
\frac{1}{(1+2n)(1+n)} & \frac{1}{3(3+2n)(3+n)}  & \frac{1}{5(5+2n)(5+n)}\\
\end{array} \right]
\] 
is nonsingular, with $det(M)=-\frac{16}{15(n+1)(n+3)(n+5)(1+2n)(3+2n)(5+2n)}$. Then, the system \eqref{S2} has only the trivial solution, namely, $a_0=a_2=a_{4}=0$. Hence, $f(\cdot)$ has only odd powers of $t$ and this finishes the proof of Corollary \ref{T4}.



\begin{thebibliography}{00}

\bibitem{AGG} {\small {\sc M.J. Alvarez, A. Gasull and H. Giacomini:} \textit{A new uniqueness criterion for the number of periodic orbits of Abel equations}, J. Differential Equations {\bf234} (2007), 161--176.}







\bibitem{AL} {\small {\sc M. A. M. Alwash and N. G. Lloyd:} \textit{Nonautonomous equations related to polynomial two-dimensional systems}, Proc. Roy. Soc. Edinburgh 105A (1987), 129-152.} 



\bibitem{AL2} {\small {\sc M.A.M. Alwash:} \textit{The composition conjecture for Abel equation}, Expo. Math. {\bf 27} (3) (2009), 241--250.}









\bibitem{BBY} {\small {\sc M. Blinov, M. Briskin, and Y. Yomdin:} \textit{Local center conditions for the Abel equation and
cyclicity of its zero solution}, in Complex Analysis and Dynamical Systems II, Contemp. Math.
382, Amer. Math. Soc., Providence, RI, 2005, pp. 65-82. MR 2007b:34079 Zbl 1095.34018}

\bibitem{Bl1} {\small {\sc M. Blinov:} \textit{Some computations around the center problem, related to the composition algebra
of univariate polynomials}, M.Sc. Thesis, Weizmann Institute of Science, 1997.}

\bibitem{Bl2} {\small {\sc M. Blinov:} \textit{Center and Composition conditions for Abel Equation}, Ph.D. thesis, Weizmann Institute
of Science, 2002.}



\bibitem{BFY1} {\small {\sc M. Briskin, J. P. Francoise and Y. Yomdin:} \textit{Center conditions, compositions
of polynomials and moments on algebraic curve}, Ergodic Theory Dyn.
Syst. {\bf19} (1999), no. 5, 1201-1220.}

\bibitem{BFY2} {\small {\sc M. Briskin, J. P. Francoise and Y. Yomdin:} \textit{Center condition II: Parametric
and model center problems}, Isr. J. Math. {\bf118} (2000), 61-82.}

\bibitem{BFY3}{\small {\sc M. Briskin, J. P. Francoise and Y. Yomdin:} \textit{Center condition III: Parametric
and model center problems}, Isr. J. Math. {\bf118} (2000), 83-108.}

\bibitem{BFY4} {\small {\sc M. Briskin, J. P. Francoise and Y. Yomdin:} \textit{Poincar\'e centre-focus problem}, C. R. Acad. Sci. Paris S\'er. I {\bf326} (1998), 1295-1298.}

\bibitem{BFY5} {\small {\sc M. Briskin, J. P. Francoise and Y. Yomdin:} \textit{Generalized moments, center-focus conditions, and compositions of polynomials, in Operator Theory}, System Theory and Related Topics (Beer-Sheva/Rehovot, 1997), Oper. Theory Adv. Appl. 123, Birkh\"auser, Basel, 2001, 161-185. MR 2002e:34050 Zbl 1075.34509}

\bibitem{BRY} {\small {\sc M. Briskin, N. Roytwarf and Y. Yomdin:} \textit{Center conditions at infinity for Abel differential equations}, Ann. Math. {\bf172} (2010), 437-83.}



\bibitem{Br1} {\small {\sc A. Brudnyi:} \textit{An algebraic model for the center problem}, Bull. Sci. math. {\bf128} (2004), 839--857.}

\bibitem{Br2} {\small {\sc A. Brudnyi:} \textit{On center sets of ODEs determined by moments of their coefficients}, Bull. Sci.
math. {\bf130} (2006), 33--48.}





\bibitem{CL}
{\small {\sc M. Carbonell and J. Llibre:} \textit{Hopf bifurcation, averaging methods and Liapunov quantities
for polynomial systems with homogeneous nonlinearities}, in Proc. Eur. Conf. on
Iteration Theory, ECIT87, 145--160 (World Scientific, Singapore, 1989).}

\bibitem{C}
{\small {\sc I. A. Cherkas:} \textit{Number of limit cycles of an autonomous second-order system}, Diff. Eqns \textbf{5} (1976),  666--668.}



\bibitem{CGM} {\small {\sc A. Cima, A. Gasull and F. Manosas:} \textit{Centers for trigonometric Abel equations}, Qual. Theory Dyn. Syst. {\bf 11} (1) (2012), 19--37.}



\bibitem{Ch} {\small {\sc C. Christopher:} \textit{Abel equations: composition conjectures and the model problem}, Bull. London Math. Soc. {\bf32} (2000), 332--338.}


\bibitem{CL2} {\small {\sc E. A. Coddington and N. Levinson:} \textit{Theory of Ordinary Differential Equations}, McGraw-Hill, New York, 1982.}


\bibitem{Du} {\small {\sc H. Dulac:} \textit{D\'etermination et integration d'une certaine classe d'\'equations diff\'erentielle ayant par point singulier un centre}, Bull. Sci. Math. S\'er. (2), {\bf 32} (1908), 230--252.}


\bibitem{GGS} {\small {\sc J. Gin\'{e}, M. Grau and X. Santallusia:} \textit{The center problem and composition condition for Abel differential equations}, Expo. Math. {\bf 34} (2016), 210-222.}

\bibitem{GGS2} {\small {\sc J. Gin\'{e}, M. Grau and X. Santallusia:} \textit{Universal centers in the cubic trigonometric Abel
equation}, Electronic Journal of Qualitative Theory of Differential Equations
(2014), No. 1, 1--7.}

\bibitem{Ka} {\small {\sc I. Kaplansky:} \textit{Commutative Rings}, Allyn and Bacon, Newton, MA, 1970.}


\bibitem{LLLZ}{\small {\sc C. Li, W. Li, J. Llibre and Z. Zhang:} \textit{On the limit cycles of polynomial
differential systems with homogeneous nonlinearities}, Proc. of the Edinburgh
Math. Soc. {\bf43} (2000), 529--543.}


\bibitem{LLZ}
{\small {\sc C. Li, W. Li, J. Llibre and Z. Zhang:} \textit{New families of centers and limit cycles for polynomial differential systems with homogeneous nonlinearities}, Ann. Differential Equations \textbf{193} (2003), 302--317.}


\bibitem{Li}
{\small {\sc A. Liapunoff:}  \textit{Probl\`{e}me g\'{e}n\'{e}ral de la stabilit\'{e} du mouvement}, Annales de la Facult\'{e} des Sciences de Touluose S\'{e}r 2 \textbf{9} (1907) 204--477. Reproduction in Annals of Mathematics Studies 17, Princeton: Princeton University Press, 1947, reprinted 1965, Kraus Reprint Corporation,
New York.}


\bibitem{LY}
{\small {\sc Y. Lijun and T. Yun:} \textit{Some New Results on Abel Equations}, J. Math. Anal. Appl. \textbf{261}, (2001) 100--112.}


\bibitem{Pa} {\small {\sc F. Pakovich:} \textit{A counterexample to the composition conjecture}, Proc. Amer. Math. Soc {\bf 130}, (2002) 3747-3749.}



\bibitem{P}
{\small {\sc H. Poincar\'{e}:} \textit{M\'{e}moire sur les courbes d\'{e}finies par une \'{e}quation diff\'{e}rentielle}, J. Math. Pures et Appl.: (S\'{e}r. 3) \textbf{7} (1881) 375--422; (S\'{e}r. 3) \textbf{8} (1882) 251--296; (S\'{e}r. 4) \textbf{1} (1885), 167--244; (S\'{e}r. 4) \textbf{2} (1886) 151--217.}





\bibitem{Re} {\small {\sc J. W. Reyn:} \textit{A bibliography of the qualitative theory of quadratic systems of differential equations in the plane}, in Report of the Faculty of Technical Mathematics and Information, Delft, 3rd edn (1994), 94--02.}










\bibitem{YT} {\small {\sc L. Yang and Y. Tang:} \textit{Some new results on Abel equations}, J. Math. Anal. Appl. {\bf261} (2001), 100-112.} 

\bibitem{Yo} {\small {\sc Y. Yomdin:} \textit{The center problem for the Abel equation, compositions of functions, and moment
conditions}, Moscow Math. J. {\bf3} (2003), 1167-1195, With an addendum by F. Pakovich.} 










\end{thebibliography}
\end{document}